\documentclass[a4paper,12pt]{article}
\usepackage{amsfonts}
\usepackage{amscd,color}
\usepackage{amsmath,amsfonts,amssymb,amscd}
\usepackage{indentfirst,graphicx,epsfig}
\usepackage{graphicx,psfrag}
\input{epsf}

\newtheorem{thm}{Theorem}[section]

\newtheorem{lem}{Lemma}[section]

\newtheorem{claim}{Claim}[section]

\newtheorem{pro}{Proposition}[section]
\newenvironment {proof} {\noindent{\em Proof.}}{\hspace*{\fill}$\Box$\par\vspace{4mm}}

\def\qed{\hfill \nopagebreak\rule{5pt}{8pt}}

\title{An extension of Mantel's theorem to random 4-uniform hypergraphs\footnote{Supported by
NSFC, the ``973" program and PCSIRT.}}

\author{Ran Gu, Xueliang Li, Zhongmei Qin, Yongtang Shi, Kang Yang \\
{\small  Center for Combinatorics and LPMC-TJKLC}\\
{\small Nankai University, Tianjin 300071, P.R. China}\\
{\small Email: guran323@163.com, lxl@nankai.edu.cn, qinzhongmei90@163.com,}\\
{\small shi@nankai.edu.cn, yangkang89@163.com} \\
}
\date{}
\begin{document}
\maketitle
\begin{abstract}
A sparse version of Mantel's Theorem is that, for
sufficiently large $p$, with high probability (w.h.p.), every
maximum triangle-free subgraph of $G(n,p)$ is bipartite. DeMarco and
Kahn proved this for $p>K \sqrt{\log n/n}$ for some constant $K$,
and apart from the value of the constant, this bound is the best
possible. Denote by $T_3$ the 3-uniform hypergraph with vertex set
$\{a,b,c,d,e\}$ and edge set $\{abc,ade,bde\}$. Frankl and F\"uredi
showed that the maximum 3-uniform hypergraph on $n$ vertices
containing no copy of $T_3$ is tripartite for $n> 3000$. For some
integer $k$, let $G^k(n,p)$ be the random $k$-uniform hypergraph.
Balogh et al. proved that for $p>K \log n/n$ for some constant $K$,
every maximum $T_3$-free subhypergraph of $G^3(n,p)$ w.h.p. is
tripartite and it does not hold when $p=0.1 \sqrt{\log n}/n$. Denote
by $T_4$ the 4-uniform hypergraph with vertex set
$\{1,2,3,4,5,6,7\}$ and edge set $\{1234,1235,4567\}$. Pikhurko
proved that there is an $n_0$ such that for all $n\ge n_0$, the
maximum 4-uniform hypergraph on $n$ vertices containing no copy of
$T_4$ is 4-partite. In this paper, we extend this type of extremal
problem in random 4-uniform hypergraphs. We show that
 for some constant $K$ and $p>K \log n/n$, w.h.p.
every maximum $T_4$-free subhypergraph of $G^4(n,p)$ is 4-partite.\\[2mm]
\end{abstract}

\section{Introduction}
Mantel's theorem \cite{Mantel} is known as a
cornerstone result in extremal combinatorics, which shows that every
triangle-free graph on $n$ vertices has at most $\lfloor
n^2/4\rfloor$ edges and the unique triangle-free graph that achieves
this bound is the  complete bipartite graph whose partite
sets are as equally-sized as possible. In other words,
every maximum (with respect to the number of edges) triangle-free
subgraph of $K_n$ is bipartite.

It is natural to generalize Mantel's theorem to
hypergraphs. The \emph{generalized triangle},
denoted by $T_k$, is a $k$-uniform
hypergraph with vertex set $[2k-1]$ and edges
$\{1,\ldots, k\}$, $\{1, 2, \ldots, k -1, k + 1\}$, and $\{k, k + 1,
\ldots , 2k -1\}$. The \emph{Tur\'{a}n hypergraph $T_r(n)$} is the
complete $n$-vertex $r$-uniform $r$-partite hypergraph whose partite
sets are as equally-sized as possible. In particular, Mantel's
Theorem states that the maximum triangle-free graph on $n$ vertices
is $T_2(n)$. Frankl and F\"{u}redi \cite{FF} proved that the maximum
3-uniform hypergraph on $n$ vertices containing no copy of $T_3$ is
$T_3(n)$ for $n > 3000$. In \cite{Pikhurko},
Pikhurko proved that there exists an $n_0$ such that for all $n\ge
n_0$, the maximum 4-uniform hypergraph on $n$ vertices containing no
copy of $T_4$ is $T_4(n)$. Actually, there are also some results on
the maximum $k$-uniform hypergraphs containing no copy of some other
specific hypergraphs, such as paths and cycles \cite{BK, FJ, MV}.

DeMarco and Kahn \cite{DK} considered a sparse version of Mantel's
Theorem. Let $G$ be the Erd\H{o}s-R\'{e}nyi random graph $G(n,p)$.
An event occurs \emph{with high probability} (w.h.p.) if the
probability of that event approaches 1 as $n$ tends to infinity. It
is interesting to determine for what $p$ every maximum triangle-free
subgraph of $G(n,p)$ is w.h.p. bipartite. DeMarco and Kahn proved
that this holds if $p> K \sqrt{\log n/n}$ for some constant $K$, and
apart from the value of the constant this bound is the best possible.
Problems of this type were first considered by Babai, Simonovits and
Spencer \cite{BSS}. Brightwell, Panagiotou, and Steger \cite{PS}
proved the existence of a constant $c$, depending only on $\ell$,
such that whenever $p \geq n^{-c}$, w.h.p. every maximum
$K_{\ell}$-free subgraph of $G(n,p)$ is $(\ell-1)$-partite.

Recently, Balogh et al. \cite{BBHL} studied an
extremal problem of this type in random 3-uniform hypergraphs. For
$n \in \mathbb{Z}$ and $p \in [0,1]$, let $G^r(n,p)$ be a random
$r$-uniform hypergraph with $n$ vertices and each element of
$\binom{[n]}{r}$ occurring as an edge with probability $p$
independently of each other. Note that in particular,
$G^2(n,p)=G(n,p)$ is the usual graph case. Balogh et
al. showed that for $p>K \log n/n$ for some constant $K$, every
maximum $T_3$-free subhypergraph of $G^3(n,p)$ w.h.p. is tripartite
and it does not hold when $p=0.1 \sqrt{\log n}/n$.

In this paper, we extend this type of  extremal
problem in random 4-uniform hypergraphs. Denote by $T$ the
4-uniform hypergraph $T_4$ with vertex set $\{1,2,3,4,5,6,7\}$ and
edge set $\{1234,1235,4567\}$.  As a sparse version
of Pikhurko's result, we obtain the following theorem.
\begin{thm}\label{th1}
There exists a positive constant $K$ such that w.h.p. the following is true. If
$G = G^4(n, p)$ is a 4-uniform random hypergraph with $p > K \log n/n$, then every maximum
$T$-free subhypergraph of $G$ is 4-partite.
\end{thm}
We should point out that the proof technique of Theorem \ref{th1} is
similar but more complicated than that of Balogh et
al. used in \cite{BBHL}. The following theorem from
\cite{Samotij} states asymptotic general structure, which is much
useful to our proof of the main result.
\begin{thm}\label{th2}
For every $\delta > 0$ there exist positive constants $K$ and $\epsilon$ such that if $p_n \ge K/n$, then w.h.p. the following holds. Every $T$-free subhypergraph of $G^4(n, p_n)$ with at least
$(3/32 -\epsilon)\binom{n}{4}p_n$ edges admits a partition $(V_1, V_2, V_3, V_4)$ of $[n]$ such that all but at most $\delta n^4p_n$
edges have one vertex in each $V_i$.
\end{thm}
Actually, asymptotic general structure statements
can also be concluded from the recent results of \cite{BMS, CFZ, CG,
Samotij}, et al..

The rest of the paper is organized as follows. In
Section 2, we introduce some more notation and prove some standard
properties of $G^4(n,p)$. The main theorem is proved in Section 3.
To simplify the formulas, we shall omit floor and ceiling signs when
they are not crucial. Undefined notation and terminology can be
found in \cite{Bondy}.

\section{Preliminaries}
Let  $G$ denote the 4-uniform random hypergraph $G^4(n, p)$. The
{\it size} of a hypergraph $H$, denoted by $|H|$,
is the number of hyperedges it contains. We denote by $q(G )$ the
size of a largest 4-partite subhypergraph of $G$. We simply write $x
= (1 \pm \epsilon)y$ when $(1 -\epsilon)y \leq x \leq (1
+\epsilon)y$. A vertex set partition $\Pi = (A_1,
A_2, A_3, A_4 )$ is {\it balanced} if $|A_i| = (1 \pm 10^{-10})n/4$
for all $i$. Given a partition $\Pi = (A_1, A_2, A_3, A_4 )$ and a
4-uniform hypergraph $H$, we say that an edge $e$ of $H$ is
\emph{crossing} if $e \cap A_i$ is non-empty for every $i$. We use
$H[\Pi]$ to denote the set of crossing edges of $H$.

The \emph{link hypergraph} $L(v)$ of a vertex $v$ in $G$ is the
3-graph with vertex set $V (G)$ and edge set $\{xyz : xyzv \in G\}$.
The \emph{crossing link hypergraph} $L_\Pi(v)$ of a vertex $v$ is
the subhypergraph of $L(v)$ whose edge set is $\{xyz$ : $xyzv$ is a
crossing edge of $G\}$. The \emph{degree} $d(v)$ of $v$ is the size
of $L(v)$, while the \emph{crossing degree}
$d_\Pi(v)$ of $v$ is the size of $L_\Pi(v)$. The \emph{common link
hypergraph} $L(u, v)$ of two vertices $u$ and $v$ is $L(u) \cap
L(v)$ and \emph{the common degree} $d(u, v)$ is the size of $L(u,
v)$. The \emph{common crossing link hypergraph} $L_\Pi(u, v)$ of two
vertices $u$ and $v$ is $ L_\Pi(u) \cap L_\Pi(v)$ and the
\emph{common crossing degree} $d_\Pi(u, v)$ is the size of $L_\Pi(u,
v)$. Given three vertices $u$, $v$ and $w$, their
\emph{co-neighborhood} $N(u, v,w)$ is the set $\{x : xuvw \in G\}$,
and the \emph{co-degree} of $u$, $v$ and $w$ is the number of
vertices in their co-neighborhood. Similarly, given two vertices $u$
and $v$, their co-neighborhood
 $N(u, v)$ is $\{xy : xyuv \in G\}$; the co-degree of $u$ and $v$ is the number of elements
in their co-neighborhood.

Given two disjoint sets $A$ and $B$, we use $[A, B]$ to denote the
set $\{a \cup b : a \in A, b \in B\}$. In this paper, $[A, B]$
usually consists of edges. Given a graph or hypergraph $G$, let
$G[A,B]$ denote the set $G \cap[A,B]$. We say a vertex partition
$\Pi$ with four classes, which we will call a 4-partition, is {\it
maximum} if $|G[\Pi]| =q(G)$. Let $F$ be a maximum $T$-free
subhypergraph of $G$. Clearly $q(G) \leq |F|$. Thus, to prove
Theorem \ref{th1}, it is sufficient to show that w.h.p. $|F|\leq
q(G)$. Moreover, we will prove that if $F$ is not 4-partite, then
w.h.p. $|F| < q(G)$.

In the following, some propositions of $G^4(n, p)$
will be stated. We will use the following Chernoff-type bound to
prove those propositions.

\begin{lem}[\cite{AS}]\label{lem3}
Let $Y$ be the sum of mutually independent indicator random variables, and let
$\mu = \mathbb{E}[Y]$, the expectation of $Y$. For all $\epsilon > 0$,
$$Pr[|Y - \mu| > \epsilon\mu] < 2e^{-c_\epsilon \mu},$$
where $c_\epsilon = \min\{-\ln (e^\epsilon(1 + \epsilon)^{-(1+\epsilon)}, \epsilon^2/2\}$.
\end{lem}
In the sequel, we use $c_\epsilon$ to denote the
constant in Lemma \ref{lem3}.
\begin{pro}\label{prop4}
For any $0<\epsilon < 1$, there exists a constant $K$ such that if $p > K\log n/n$, then
w.h.p. the co-degree of any triple of vertices in $G$ is $(1\pm \epsilon)pn$.
\end{pro}
\begin{proof}
For each triple of vertices $u,v,w$, let $X_{u,v,w}$ denote the
number of vertices $x \in  V \setminus \{u,v,w\}$ such that $xuvw$
is an edge. Letting $\mu = \mathbb{E}[X_{u,v,w}]$, we have $\mu= p(n
-3)$, and by Lemma \ref{lem3},
$$Pr[|X_{u,v,w} -\mu| > \epsilon \mu] < 2e^{-c_\epsilon(n-3)p} < 2e^{-c_\epsilon \frac{np}{2}}.$$
If $K > 8/c_\epsilon$, then $e^{-c_\epsilon \frac{np}{2}} < n^{-4}$. By the union bound, it  follows that with probability at most $n^3n^{-4} = n^{-1}$, the event
 $|X_{u,v,w} - \mu| > \epsilon \mu$ holds for some $\{u,v,w\}$. Therefore, w.h.p. there
is no such $\{u,v,w\}$.
\end{proof}

\begin{pro}\label{prop5}
For any $0<\epsilon < 1$, there exists a constant $K$ such that if $p > K \log n/n$, then
w.h.p. the co-degree of any pair of vertices in $G$ is $(1\pm \epsilon)\frac{p}{2}n^2$.
\end{pro}
\begin{proof}
For each pair of vertices $u$ and $v$, let
$X_{u,v}$ be the random variable given by the number of pairs of
vertices $x,y \in  V \setminus \{u,v\}$ such that $xuvy$ is an edge.
Letting $\mu = \mathbb{E}[X_{u,v}]$, we have $\mu = p\binom{n -
2}{2}$, and then by Lemma \ref{lem3},
$$Pr[|X_{u,v} - \mu| > \epsilon \mu] < 2e^{-c_\epsilon p\binom{n - 2}{2}} <2 e^{-c_\epsilon \frac{n^2p}{4}}.$$
If $K > 8/c_\epsilon$, then $e^{-c_\epsilon \frac{np}{2}} < n^{-4}$. By the union bound, it  follows that the probability
that $|X_{u,v,w} - \mu| > \epsilon \mu$ for some $\{u,v,w\}$ is at most $n^22e^{-c_\epsilon \frac{n^2p}{4}}=o(1)$.
\end{proof}

\begin{pro}\label{prop6}
For any $0<\epsilon < 1$, there exists a constant $K$ such that if $p > K \log n/n$, then
w.h.p. the common degree of any pair of vertices in $G$ is $(1\pm \epsilon)\frac{p^2}{6}n^3$.
\end{pro}
\begin{proof}
For two disjoint sets of vertices $\{u,v\}$ and
$\{a,b,c\}$, let $\mathcal{A}_{u,v}^{a,b,c}$ be the event $\{uabc
\in G, vabc \in G\}$, and let $X_{u,v}^{a,b,c}$
 be the indicator random variable of $\mathcal{A}_{u,v}^{a,b,c}$. Then
$d(u,v)=\sum\limits_{abc}{\mathcal{A}_{u,v}^{a,b,c}}$.
 Letting $\mu = \mathbb{E}[d(u,v)]$, we
have $\mu =p^2\binom{n - 2}{3}$ , and by Lemma \ref{lem3},
$$Pr[|d(u,v)-\mu| > \epsilon \mu] < 2e^{-c_\epsilon p\mu} < e^{-c_\epsilon \frac{n^3p^2}{12}}.$$
By the union bound, it  follows that the probability
that $|d(u,v) - \mu| > \epsilon \mu$ for some $\{u,v\}$ is at most $n^2e^{-c_\epsilon \frac{n^3p^2}{12}}=o(1)$.
\end{proof}

\begin{pro}\label{prop7}
For any $0<\epsilon < 1$, there exists a constant $K$ such that if $p > K \log n/n$, then
w.h.p.  for any vertex $v$ of $G$, its degree $d(v)$  is $(1\pm \epsilon)\frac{p}{6}n^3$.
\end{pro}
\begin{proof}
For each  vertex $v$, let $X_{v}$ be the random variable given
by the number of triples of  vertices $x,y,z \in  V \setminus \{v\}$ such that $vxyz$ is an edge. Letting $\mu = \mathbb{E}[X_{v}]$, we
have $\mu = p\binom{n - 1}{3}$ , and by Lemma \ref{lem3},
$$Pr[|X_{v} - \mu| > \epsilon \mu] < 2e^{-c_\epsilon p \mu} .$$
 By the union bound, it therefore follows that the probability
that $|d(v) - \mu| > \epsilon \mu$ holds for some $v$ is at most $n2e^{-c_\epsilon \mu}=o(1)$.
\end{proof}

\begin{pro}\label{prop8}
For any $0<\epsilon < 1$, there exists a constant $K$ such that if $p > K \log n/n$, then
w.h.p.  for any 4-partition $\Pi=(A_1,A_2,A_3,A_4)$ with $|A_2|$, $|A_3|$, $|A_4|\geq\frac{n}{80}$, and any vertex $v\in A_1$  we have $d_\Pi(v)=(1\pm \epsilon)p|A_2||A_3||A_4|$.
\end{pro}
\begin{proof}
Since $|A_2|, |A_3|,|A_4|\ge n/80$, so $|A_2||A_3||A_4|\ge n^3/80^3$. Let $\mu = \mathbb{E}[d_{\Pi}(v)] =|A_2||A_3||A_4|p$. We have
$$Pr[|d_{\Pi}(v) - \mu|> \epsilon \mu] < 2e^{-c_\epsilon \frac{n^3p}{80^3}}.$$
Then by the union bound, the probability that the
statement does not hold is bounded by
$$4^nn Pr[|d_{\Pi}(v) - \mu|> \epsilon \mu]  < e^{3n + \log n -c_\epsilon \frac{n^3p}{80^3}}=o(1).$$
\end{proof}

For a vertex $v$ and a vertex set $S$, let
$\mathcal{E}$ be a subset of $\{vxab \in G : x \in S \}$ satisfying
that for any $x \in S$, there exists $W \in
\mathcal{E}$ such that $x \in W$. Let $Q$  be a
subset of $L(v)$. Define $G_{v,\mathcal{E}}[S, Q ] = \{xuwz \in G :
x \in  S, uwz \in Q, \exists W \in \mathcal{E}$ $ s.t. \ x \in W,\
u,w,z \notin W \}$. Then for any $xuwz\in G_{v,\mathcal{E}}[S, Q ]$
with $x \in S$ and $uwz \in Q$, we can find a $T = \{vuwz, xuwz,
vxab\}$, where $vxab \in \mathcal{E}$. The condition $u,w,z \notin W
$ in the definition of $G_{v,\mathcal{E}}[S, Q ]$ guarantees that we
can find such a $T$.

\begin{pro}\label{prop9}
For any $0<\epsilon$,  $\epsilon_1$, $\epsilon_2< 1$, there exists a
constant $K$ such that if $p > K \log n/n$, then w.h.p.  for any
choice of  $\{v,S,\mathcal{E},Q\}$ as above with $|S|=\epsilon_1n$,
$|Q|=\epsilon_2pn^3$, we have $|G_{v,\mathcal{E}}[S, Q ]|=(1\pm
\epsilon)p|S||Q|$.
\end{pro}
\begin{proof}
Let $K_{v,\mathcal{E}}[S, Q ] = \{xuwz : x \in  S, uwz \in Q,
\exists W \in \mathcal{E}$ $ s.t. x \in W,\ u,w,z \notin W \}$. Then
for $K_{v,\mathcal{E}}[S, Q ]$ and $\mathbb{E}[G_{v,\mathcal{E}}[S,
Q ] ]$, we have the relation as follows.
$$\mathbb{E}[G_{v,\mathcal{E}}[S, Q ] ]=p |K_{v,\mathcal{E}}[S, Q ]|.$$
For $x \in S$, let $d_{\mathcal{E}}(x)=|\{W \in \mathcal{E}:x \in
W\}|$ and $Q_x=\{abc \in Q: vxay \in \mathcal{E}\ for\ some\ vertex\
y\}$. If $d_{\mathcal{E}}(x)>6pn$, then clearly $[x,Q]\subseteq
K_{v,\mathcal{E}}[S, Q ]$. If $d_{\mathcal{E}}(x)\leq 6pn$, then by
Proposition \ref{prop5} we have $|Q_x|\leq 12pn \times 2
\frac{pn^2}{2}=12p^2n^3$. Clearly $[x,Q\backslash Q_x] \subseteq
K_{v,\mathcal{E}}[S, Q ]$. Hence, $$|[S,Q]|-|K_{v,\mathcal{E}}[S, Q
]|\leq \sum_{x \in S, d_{\mathcal{E}}(x)\leq 6pn}|Q_x|\leq |S|\times
12p^2n^3 = 12\epsilon_1p^2n^4.$$ On the other hand, we have
$|[S,Q]|=|S||Q| = \epsilon_1\epsilon_2pn^4$, so
$|K_{v,\mathcal{E}}[S, Q ]|=(1-o(1))|S||Q|$. Let
$\mu=\mathbb{E}[G_{v,\mathcal{E}}[S, Q ] ]=p|K_{v,\mathcal{E}}[S, Q
]|=(1-o(1))p|S||Q|$. By Lemma \ref{lem3} we have
$$Pr[||G_{v,\mathcal{E}}[S, Q ] |-\mu|> \epsilon \mu] <2e^{-c_\epsilon\mu}.$$
We have at most $n$ choices for $v$, $\binom{n}{\epsilon_1n}$ choices for $S$, $2^{pn^2\cdot \epsilon_1n}$ choices for $\mathcal{E}$ and $\binom{\frac{1}{3}pn^3}{\epsilon_2pn^3}$ choices for $Q$. Then by the union bound, the probability that the statement of Proposition \ref{prop9} does not hold is bounded by
$$n\binom{n}{\epsilon_1n}2^{pn^2\epsilon_1n}
\binom{\frac{1}{3}pn^3}{\epsilon_2pn^3}2e^{-c_\epsilon \epsilon_1\epsilon_2  p^2n^4} =o(1).$$
\end{proof}

\begin{pro}\label{prop10}
For any $0<\epsilon < 1$, there exists a constant $K$ such that if $p >
K \log n/n$, then w.h.p. the following holds: if $F$ is a maximum
$T$-free subhypergraph of $G$ and $\Pi$ is a 4-partition
 maximizing $|F[\Pi]|$, then $|F| \geq (\frac{3}{32} -\epsilon)\binom{n}{4}p$, and $\Pi$ is a balanced partition.
\end{pro}
\begin{proof}
It suffices to prove it when $\epsilon$ is small. For a partition
$\Pi=(A_1,A_2,A_3,A_4)$, it is clear that  $|F| \geq q(G) \geq
|G[\Pi]|$. And Proposition \ref{prop8} implies that w.h.p.
$|G[\Pi]|=(1\pm \epsilon)p|A_1||A_2||A_3||A_4|$ if $|A_2|$, $|A_3|$,
$|A_4| \geq \frac{n}{80}$. Consider a partition satisfying
$|A_1|=|A_2|=|A_3|=|A_4|=\frac{n}{4}$, then we have $|F| \geq
(\frac{3}{32}-\epsilon)\binom{n}{4}p$. Also, Theorem \ref{th2}
implies that if $\Pi$ maximizes $|F[\Pi]|$, then $|G[\Pi]|\geq
|F[\Pi]| \geq (\frac{3}{32}-2\epsilon)\binom{n}{4}p$.

If $\Pi$ is not a balanced partition and $|A_2|$, $|A_3|$, $|A_4| \geq \frac{n}{80}$, then $|G[\Pi]| \leq (1+ \epsilon)p|A_1||A_2||A_3||A_4|< (\frac{3}{32}-2\epsilon)\binom{n}{4}p$. If $\Pi$ is not balanced and one of $|A_1|, |A_2|, |A_3|, |A_4|$ is less than $ \frac{n}{80}$, then Proposition \ref{prop7} implies that $|G[\Pi]|< \frac{n}{80} (1+ \epsilon)\frac{1}{6}pn^3<(\frac{3}{32}-2\epsilon)\binom{n}{4}p $. Therefore, if $\Pi$ maximizes $|F[\Pi]|$, then $\Pi$ is balanced.
\end{proof}

Let $\alpha = 0.35$. Given a balanced partition
$\Pi=(A_1,A_2,A_3,A_4)$, let $P(\Pi) = \{(u, v) \in \binom{A_1}{2} :
d_\Pi(u, v) < \frac{\alpha}{32}p^2n^3\}$. In other
words, $P(\Pi)$ is the set of pairs of vertices in $A_1$ that have
low common crossing degree.

\begin{pro}\label{prop11}
There exists a constant $K$ such that if $p > K \log n/n$, then
w.h.p. for every balanced partition $\Pi$, every vertex $v$ and every positive constant $\xi > 0$, we have $d_{P(\Pi)}(v) < \frac{\xi}{p}$, where $d_{P(\Pi)}(v)$ denotes the number of elements containing $v$ in $P(\Pi)$.
\end{pro}
\begin{proof}
Let $\epsilon=0.1$. By Proposition \ref{prop8}, we assume that $d_{\Pi}(v)\geq (1-\epsilon)p\frac{n^3}{64}$, and therefore, $d_{\Pi}(u,v)< \frac{2\alpha}{1-\epsilon}d_{\Pi}(v)p$ for $(u,v) \in d_{P(\Pi)}(v)$.

If a vertex $v$ and a balanced cut $\Pi$ violate the statement of Proposition \ref{prop11}, then there are $S \subseteq V$ and $Q=L_{\Pi}(v)$ with $|S|:=s=\lceil\xi/p\rceil$ and $|G[S,Q]|\leq \frac{2\alpha}{1-\epsilon}|S||Q|p$. We have at most $4^n$ choices of $\Pi$, $n$ choices of $v$, $\binom{n}{s}$ choices of $S$, so the probability of such a violation is at most
$$4^n n\binom{n}{s}e^{-c\cdot\frac{\xi}{p}\cdot pn^3\cdot p} $$
for some small constant $c$, and therefore is $o(1).$
\end{proof}

The following lemma plays an important role in the proof of Lemma \ref{lem13}.

\begin{lem}\label{lem12}
Let $\beta$ and $r$ be positive integers. For any $\epsilon > 0$, there exists a constant $K$ such that if $p > K \log n/n$, $\beta\leq\epsilon n$ and
\begin{equation}\label{eq1}
\binom{n}{\beta}\binom{{n^3}}{r}\exp(-c_{\epsilon}\epsilon pnr)=o(1).
\end{equation}
Then w.h.p. the following holds: for any set of
vertices $A$ with $|A| \leq \beta$, there are at most $r$ triples
$\{u, v,w\}\in \binom{V (G)}{3}$ such that $|N(u, v, w) \cap A| >
2\epsilon pn$.
\end{lem}

\begin{proof}
Fix a vertex set $A$ of size $\beta$. We shall show that there are at most $r$ triples $\{u, v,w\}\in \binom{V (G)}{3}$ such that $|N(u, v, w) \cap A|$ is large. For each triple $\{u, v,w\}$, let $\mathcal{B}(u, v,w)$ be the event that
$|N(u, v, w) \cap A| > 2\epsilon pn$. By Lemma \ref{lem3},
\begin{equation*}
Pr[\mathcal{B}(u, v,w)]<e^{-c\epsilon pn},
\end{equation*}
for $c=c_1$ in Lemma \ref{lem3}. If $\{u, v,w\}\neq \{u', v',w'\}$, then $\mathcal{B}(u, v,w)$ and $\mathcal{B}(u', v',w')$ are independent events.
Consequently, the probability that $\mathcal{B}(u, v,w)$ holds for at least $r$ triples  is at most
\begin{equation*}
\binom{{n^3}}{r}\exp(-c_{\epsilon}\epsilon pnr).
\end{equation*}
There are $\binom{n}{\beta}$ choices of $A$. Therefore, if
Eq. (\ref{eq1}) holds, then w.h.p. there are at
most $r$ triples $\{u, v,w\}\in \binom{V (G)}{3}$ such that $|N(u,
v, w) \cap A|> 2\epsilon pn$.
\end{proof}

\section{Proof of Theorem \ref{th1}}
Let $F$ be a $T$-free subhypergraph of $G$. We want to show that
$|F|\leq q(G)$.  We first state two key lemmas. The
first lemma proves $|F|\leq q(G)$ with some additional conditions on
$F$. Define the \emph{shadow graph} of a hypergraph $H$ is a simple
graph with $xy$ an edge if and only if there exists some edge of $H$
that contains both $x$ and $y$.

\begin{lem}\label{lem13}
Let $F$ be a $T$-free subhypergraph of $G$ and
$\Pi=(A_1,A_2,A_3,A_4)$ be a balanced partition maximizing
$|F[\Pi]|$. Denote  the shadow graph of $F$ by $F_S$. For $1 \leq i
\leq 4$, let $B_i = \{e \in F : |e \cap A_i|\geq 2\}$. There exist
positive constants $K$ and $\delta$ such that if $p > K \log n/n$
and the following conditions hold:
\begin{description}
           \item[(i)] $|\bigcup\limits_{i = 1}^4 {{B_i}} |  \le \delta p{n^4},$
           \item[(ii)] $B_1\neq\emptyset$,
           \item[(iii)] the subgraph of  $F_S$ induced by $B_1$ is disjoint from $P(\Pi)$,
\end{description}
then w.h.p. $|F[\Pi]|+ 4|B_1| < |G[\Pi]|$.
\end{lem}
{\em Remark.} We point out that if Condition (ii) does not hold,
i.e., $|B_1| = 0$, then clearly $|F[\Pi]|+ 4|B_1| \leq |G[\Pi]|$.
Therefore, Conditions (i) and (iii) imply that $|F[\Pi]|+ 4|B_1|
\leq |G[\Pi]|$, while Condition (ii) implies the
strict inequality.

Let $F_0$ be a maximum $T$-free subhypergraph of $G$. By Theorem
\ref{th2} and Proposition \ref{prop10},
w.h.p. Condition (i) of Lemma \ref{lem13} holds for every $\delta > 0$. Without loss of
generality, we may assume that $|B_1| \geq |B_2|, |B_3|, |B_4|$. If
$F_0$ is not 4-partite, then Condition (ii) of Lemma \ref{lem13}
holds. Moreover, if $P(\Pi)=\emptyset$, then Condition (iii) of
Lemma \ref{lem13} also holds. Therefore, if $P(\Pi)=\emptyset$ and
$F_0$ is not 4-partite, then we can apply Lemma \ref{lem13} to $F_0$ and
get that $|F_0| < q(G)$, a contradiction. Hence if $P(\Pi)=\emptyset$ for
every balanced partition $\Pi=(A_1,A_2,A_3,A_4)$, then the proof
would be completed. Thus, we need to consider the
case that $P(\Pi)\neq\emptyset$.  The following lemma  tells us that
if $P(\Pi)\neq\emptyset$, then $\Pi$ is far from being a maximum
4-partition.

\begin{lem}\label{lem14}
There exist positive constants $K$ and $\delta$ such that if $p > K \log n/n$, the 4-partition $\Pi$ is balanced, and $P(\Pi)\neq\emptyset$, then w.h.p.
\begin{equation}\label{eqa1}
q(G) > |G[\Pi]| + |P(\Pi)|\delta n^3p^2.
\end{equation}
\end{lem}

\noindent{\em Remark.} If $P(\Pi)=\emptyset$, then clearly $q(G)
\geq |G[\Pi]| + |P(\Pi)|\delta n^3p^2$. Therefore, Lemma \ref{lem14}
also implies that $q(G) \geq |G[\Pi]| + |P(\Pi)|\delta n^3p^2$ for
every balanced 4-partition $\Pi$. We point out that the constant
$\delta$ in  Lemma \ref{lem13} also satisfies Lemma \ref{lem14},
which can be verified in the proof of Lemma \ref{lem14}.

Now we use Lemmas \ref{lem13} and
\ref{lem14} to prove Theorem \ref{th1}. The proofs of  those two lemmas are presented in the  next two subsections.

Let $\tilde{F}$ be a maximum $T$-free subhypergraph of $G$, so
$|\tilde{F}|\ge q(G )$. To prove Theorem \ref{th1}, it is sufficient
to show that $|\tilde{F}|\le q(G )$. Let $\Pi = (A_1, A_2, A_3,A_4)$
be a 4-partition maximizing $|\tilde{F}[\Pi]|$. From
Proposition \ref{prop10}, we get that $\Pi$ is balanced. For $1 \le
i \le 4$, let $\tilde{B}_i = \{e \in \tilde{F} , |e \cap A_i | \ge
2\}$. Without loss of generality, we may assume $|\tilde{B}_1| \ge
|\tilde{B}_2|, |\tilde{B}_3|, |\tilde{B}_4|$. Let $B(\Pi) = \{e \in
G : \exists(u, v) \in P(\Pi)$ s.t. $\{u, v\} \subset e\}$ and $F
=\tilde{F}-B(\Pi)$. Then $F $ satisfies Condition (iii) of Lemma
\ref{lem13} and $\Pi$ maximizes $|F[\Pi]|$ as well. By Proposition
\ref{prop5}, we know that w.h.p. $|B(\Pi)| \le |P(\Pi)|pn^2$.
Together with  Proposition \ref{prop11}, we have that w.h.p.
$|B(\Pi)| \le O(n^3)$.  Combining with Proposition \ref{prop10}, we
can apply Theorem \ref{th2} to $F $ (let $\delta$ in Theorem \ref{th2} be  the constant $\delta$ given in
Lemma \ref{lem13}). And  since $\Pi$ maximizes $|F[\Pi]|$, we can
derive that $F $ and $\Pi$ satisfy Condition (i) of Lemma
\ref{lem13}. For $1 \le i \le 4$, let $B_i = \{e \in F , |e \cap A_i
| \ge 2\}$. Then we have:
\begin{align*}
|\tilde{F}|  &\le |\tilde{F}[\Pi]|+4|\tilde{B}_1|\\
\null  &=|F[\Pi]|+4
|{B}_1|+4|\tilde{F}\cap B(\Pi)| \\
\null  &\le |G[\Pi]|+4|\tilde{F}\cap B(\Pi)|\\
\null  &\le |G[\Pi]|+4| B(\Pi)|\\
\null  &\le |G[\Pi]|+4|P(\Pi)|\delta p^2n^3 \\
\null  &\le q(G),
\end{align*}
where $\delta$ is  the  constant  in  Lemma \ref{lem13}. The second
inequality follows from Lemma \ref{lem13} and its remark. Note that
the  constant $\delta$ in  Lemma \ref{lem13} satisfies Lemma
\ref{lem14}, the last inequality holds because of Lemma \ref{lem14} and its remark.
Hence, $|\tilde{F}| =q(G)$. So equalities hold throughout the above
inequalities. If $B_1\neq \emptyset$, then by Lemma \ref{lem13},
$|F[\Pi]|+4 |{B}_1|+4|\tilde{F}\cap B(\Pi)| <|G[\Pi]|+4| B(\Pi)|$
and if $P(\Pi)\neq \emptyset$, then $|G[\Pi]|+4|P(\Pi)|\delta p^2n^3
< q(G)$, both are contradictions. So both $B_1$ and
$P(\Pi)$ are empty sets. It follows that $\tilde{B}_1$ is an empty
set. Since we assume that $|\tilde{B}_1| \ge
|\tilde{B}_2|,|\tilde{B}_3|,|\tilde{B}_4|$, we have that
$|\tilde{B}_1| =|\tilde{B}_2|=|\tilde{B}_3|=|\tilde{B}_4|=0$, which
implies that  $\tilde{F}$ is 4-partite.

The proof is thus completed.\qed

\subsection{Proof of Lemma \ref{lem13}}
Let $\epsilon_1=\frac{1}{4200}$, $\epsilon_2=\frac{1}{7200}$, $\delta=\frac{\epsilon_1^3\epsilon_2}{320\times110\times16}$, $\epsilon_3=\frac{16\times80\delta}{\epsilon_1}$.

Let $M$ be the set consisting of crossing edges in $G \setminus F$. To prove Lemma \ref{lem13}, it suffices to prove that
$4|B_1| < |M|$, so we assume for contradiction that $|M|\leq 4|B_1|  \le 4\delta p{n^4}$. Our aim is to derive a lower bound for $|M|$, which contradicts to our assumption.

For each edge $W=w_1w_2w_3w_4\in B_1$, with $w_1,w_2\in A_1$, from Condition (iii), we have $w_1w_2\notin P(\Pi)$, which implies there are at least $\frac{\alpha}{32} p^2n^3$ choices of $x\in A_2$, $y\in A_3$, and $z\in A_4$ such that $w_1xyz$ and $w_2xyz$ are both crossing
edges of $G$. By Proposition \ref{prop4}, the co-degree of any three vertices is w.h.p. at most $2pn$. Hence,
there are at least $\frac{\alpha}{32} p^2n^3-6pn\geq\frac{1}{80}p^2n^3$ choices of such triple $(x,y, z)$ satisfying  $w_3, w_4\notin \{x,y, z\}$,
and then each of these triples $(x,y, z)$ together with $W$ form a copy of $T = \{w_1w_2w_3w_4,w_1xyz,w_2xyz\}$
in $G$. Since $F$ contains no copy of $T$, at least one of $w_1xyz$ and $w_2xyz$ must be in $M$.

We can obtain a lower bound for $|M|$ by counting the number of $T$ in $G$ that
contain some edge $W\in B_1$, since each such $T$ must contain at least one edge in $M$.

Unfortunately, it is not easy to count the number of $T$ directly. Actually, we will count copies of $\hat{T}$ instead of $T$ in $G$, where $\hat{T}$ is a 4-graph on 5 vertices $\{w_1,w_2,x,y, z\}$ with two crossing edges $w_1xyz$, $w_2xyz$ such that  there exists $W\in B_1$ with $w_1,w_2\in W\cap A_1$ and $x,y, z\notin W$. It is easy to see that every $\hat{T}$ yields many copies of $T$ containing some edge $W\in B_1$. From the previous argument, we know that for every pair $w_1,w_2\in W\cap A_1$ taken from some edge $W\in B_1$, there are at least $\frac{1}{80}p^2n^3$ $\hat{T}$'s containing $w_1,w_2$.

Let $L$ be the subgraph of $F_S$ induced by  the vertex set $A_1$,
let $C=\{x \in A_1:d_L(x)\geq\epsilon_1n\}$ and $D=A_1\setminus C.$
Let $C_1$ be the set of $x\in C$ such that there are at least
$\epsilon_2pn^3$ crossing edges in $F$ containing $x$. Set
$C_2=C\setminus C_1$.

We first prove four claims, which will be needed in the proof of Lemma \ref{lem13}.

\begin{claim}\label{lem15}
$|C|\leq \epsilon_3n$.
\end{claim}

\begin{proof}
 Since every vertex  in $C$ has degree at least $\epsilon_1n$ in $L$, we have $\epsilon_1n|C| \leq 2|E(L)| $.

On the other hand, we show  that $|E(L)|\leq 640\delta n^2$.
For each edge $wu\in E(L)$, since $wu\notin P(\Pi)$,  there are at
least $\frac{1}{80} p^2n^3$ choices of $x\in A_2$, $y\in A_3$, $z\in
A_4$, such that $\{w,u,x,y,z\}$ is the vertex set  of a $\hat{T}$ in
$G$. So at least one of $wxyz$ and $uxyz$ must be in $M$. By
Proposition \ref{prop4}, $|M|\cdot2np\geq |E(L)|\frac{1}{80}
p^2n^3$. Since we assume $|M|\leq 4|B_1|  \le 4\delta p{n^4}$, it
follows that $|E(L)|\leq 640\delta n^2$. Thus, $|C| \leq
\frac{16\times80 \delta n}{\epsilon_1}=\epsilon_3n$.
\end{proof}
The next three claims provide lower bounds for $|M|$ with different parameters.
\begin{claim}\label{lem16}
$|M|\geq \frac{{\epsilon_1}\epsilon_2}{16\epsilon_3}pn^3|C_1|$.
\end{claim}
\begin{proof}
This inequality is trivial if $|C_1|=0$, so we assume $|C_1| \geq 1$.
For each $w \in C_1$, let $Q_w =
\{(x, y,z) \in A_2 \times A_3 \times A_4 : wxyz \in F\}$. We have the number of neighbors of $w$ in $L$, denoted by $|N_L(w)|$, satisfying $|N_L(w)|\geq \epsilon_1 n$ and $|Q_w| \geq\epsilon_2pn^3$. We will count the number of copies of $\hat{T}$ with the vertex set $\{w,u,x,y,z\}$ in $G$ such that $w \in C_1$, $u\in |N_L(w)|$, $wxyz \in F$ and $uxyz \in G$. By Proposition  \ref{prop9} with $v=w$, $S\subseteq N_L(w)$ and $|S|=\epsilon_1 n$, $\mathcal{E}=\{W\in B_1:w\in W\}$
and $Q\subseteq Q_w$ with $|Q| =\epsilon_2pn^3$, there are at least $\frac{1}{2}|S||Q|p$ $\hat{T}$'s for each $w\in C_1$. Thus, the total number of such copies of $\hat{T}$ is at least
\begin{equation}\label{eq2}
\frac{1}{2}\sum\limits_{w \in {C_1}} {\frac{1}{2}|S||Q|p}= \frac{1}{4}|C_1|\epsilon_1n\cdot\epsilon_2pn^3\cdot p=\frac{1}{4}\epsilon_1\epsilon_2p^2n^4|C_1|.
\end{equation}

We call an edge $abcd \in M$  {\it bad} if $a\in
A_1$, $b\in A_2$, $c\in A_3$, $d\in A_4$ and there are at least
$2\epsilon_3pn$ vertices $x\in C_1$ such that $xbcd\in G$. Since
$|C_1|\leq|C|\leq \epsilon_3n$ (Claim \ref{lem15}), from Lemma
\ref{lem12} with $\epsilon=\epsilon_3$, $\beta=\epsilon n$,
$r=\frac{\log\log n}{p}$, $A=C_1$,  there are at most
$\frac{\log\log n}{p}$ triples $(b,c,d) \in A_2\times A_3\times A_4$
that are in some bad edge. By Proposition \ref{prop4},  each such
$(b,c,d)$ is in at most $\binom{2pn}{2}$ $\hat{T}$'s. Thus, the
number of copies of $\hat{T}$ estimated in (\ref{eq2}) that contain
a non-bad edge from $M$ is at least
\begin{equation*}
\frac{1}{4}\epsilon_1\epsilon_2p^2n^4|C_1|-\binom{2pn}{2}\frac{\log\log n}{p}.
\end{equation*}
Since $\binom{2pn}{2}\frac{\log\log n}{p}\leq2p^2n^2\frac{\log\log
n}{p}=2pn^2\log\log n\leq
\frac{1}{8}\epsilon_1\epsilon_2p^2n^4|C_1|$,  there are at least
$$\frac{1}{4}\epsilon_1\epsilon_2p^2n^4|C_1|-
\frac{1}{8}\epsilon_1\epsilon_2p^2n^4|C_1|=
\frac{1}{8}\epsilon_1\epsilon_2p^2n^4|C_1|$$ copies of $\hat{T}$
estimated in (\ref{eq2}) that contain a non-bad edge from $M$. By
the definition of bad edges, every such non-bad edge in $M$ is in at
most $2\epsilon_3pn$ such copies of $\hat{T}$, so we have
$$|M|\geq\frac{1}{8}\frac{\epsilon_1\epsilon_2p^2n^4|C_1|}{2\epsilon_3pn}
=\frac{\epsilon_1\epsilon_2}{16\epsilon_3}pn^3|C_1|.$$
\end{proof}

The following claim gives us another lower bound for $|M|$, its proof is similar to Claim \ref{lem16}, but we need more complicated analysis.

\begin{claim}\label{lem17}
Let $L'$ be a subgraph of $L$ such that $\Delta(L')\leq \epsilon_1n$, then $$|M|\geq \frac{p{n}^2}{320\epsilon_1}|E(L')|.$$
\end{claim}

\begin{proof}
Since for each $wu\in E(L')$, $wu\notin P(\Pi)$,  there are at least $\frac{1}{80} p^2n^3$ choices of $x\in A_2$, $y\in A_3$,  $z\in A_4$, such that $\{w,u,x,y,z\}$ is the vertex set  of a $\hat{T}$ in $G$. So the total number of such $\hat{T}$ is at least $\frac{|E(L')|}{80}p^2n^3$. And for each of these copies of $\hat{T}$, there must be  at least one of $wxyz$, $uxyz$  in $M$.

For an edge $xabc\in M$ with $x\in V(L')$, we will count copies of $\hat{T}$ containing edge $xabc$  in $G$. Call an edge $xabc$  bad if there exist at least $2\epsilon_1pn$ $y\in N_{L'}(x)$ with $yabc\in G$. Combining with Proposition \ref{prop5}, there exist at most $\min\{2pn, d_{L'}(x)\}$ vertices $y\in N_{L'}(x)$ with $yabc\in G$. Denote by $r_x$ the number of $(a,b,c)$ such that $xabc$ is bad. Therefore, the number of
copies of $\hat{T}$ that contain a non-bad edge from $M$ is at least
\begin{equation}\label{eq3}
\frac{1}{2}\sum\limits_{x \in V(L')} {d_{L'}(x)\frac{{{p^2}{n^3}}}{{80}}}-\sum\limits_{x \in V(L')} {r_x\min\{2pn, d_{L'}(x)\}}.
\end{equation}
We will prove $\frac{1}{4} {d_{L'}(x)\frac{{{p^2}{n^3}}}{{80}}}\geq
{r_x\min\{2pn, d_{L'}(x)\}}$ for every vertex $x \in V(L')$.
Depending on the value of $d_{L'}(x)$, we have three cases.

\noindent{\bf Case 1.} $d_{L'}(x)> 2pn$. $\frac{\log n}{p^2n^2}\leq
d_{L'}(x)\leq\epsilon_1n$, We apply Lemma \ref{lem12} with
$\beta=\epsilon_1n$ and $r = \frac{\log \log n}{p}$ to obtain that
$r_x \leq \frac{\log \log n}{p}$.

\noindent{\bf Case 2.} $d_{L'}(x)> 2pn$ and $\frac{\log
n}{p^{k+1}n^{k+1}}\leq d_{L'}(x)\leq\frac{\log n}{p^{k}n^{k}}$  for
some integer $k \in [2,\frac{\log n}{\log \log n}]$. We apply Lemma
\ref{lem12} with $\beta=\frac{\log n}{p^{k}n^{k}}$ and $r =
\beta/100$ to obtain that $r_x \leq\frac{\log n}{100p^{k}n^{k}}\leq
\frac{pnd_x}{100}$.

\noindent{\bf Case 3.} $d_{L'}(x)\leq 2pn$. We apply Lemma
\ref{lem12} with $\beta=2pn $ and $r =\frac{p^2n^2}{100}$ to obtain
that $r_x \leq \frac{p^2n^2}{100}$.

One can easily check that, for each of the above
three cases,
$$\frac{1}{4} {d_{L'}(x)\frac{{{p^2}{n^3}}}{{80}}}\geq {r_x\min\{2pn, d_{L'}(x)\}}.$$
Therefore, the number of copies of $\hat{T}$ estimated in (\ref
{eq3}) is at least $$\frac{1}{4}\sum\limits_{x \in V(L')}  {d_{L'}(x)\frac{{{p^2}{n^3}}}{{80}}}
=\frac{1}{2}\frac{|E(L')|}{80}p^2n^3.$$
Bearing in mind that a non-bad edge is in at most $2\epsilon_1pn$ copies of
$\hat{T}$ estimated in (\ref{eq3}), thus, we have
$$|M|\geq \frac{1}{2}\frac{|E(L')|}{80}p^2n^3\cdot\frac{1}{2\epsilon_1pn}
=\frac{|E(L')|}{320\epsilon_1}pn^2.$$
\end{proof}

\begin{claim}\label{lem18}
$|M|\geq \frac{p{n}^3}{130}|C_2|.$
\end{claim}

\begin{proof}
For every vertex $x \in C_2$, the number of edges in $F [\Pi]$  containing $x$ is at most $\epsilon_2pn^3$.
On the other hand,  by Proposition \ref{prop8}, w.h.p. the crossing degree of $x$ in $G$ is at least $\frac{p{n}^3}{65}$. Thus,
there are at least $\frac{p{n}^3}{65}-\epsilon_2pn^3\ge \frac{p{n}^3}{130}$ edges of $M$ incident to $x$, so $|M|\geq \frac{p{n}^3}{130}|C_2|.$
\end{proof}

Next we present the proof of Lemma \ref{lem13}.
First we divide the edges in $B_1$ into three classes.
Set
$$B_1^{(1)}=\{e\in B_1: |e\cap C|\leq 3 , \ |e\cap C_1|\geq 1\},$$
$$B_1^{(2)}=\{e\in B_1\setminus B_1^{(1)}: |e\cap C|\leq 3,|e\cap C_2|\geq 1\},$$
$$B_1^{(3)}=\{e\in B_1\setminus (B_1^{(1)}\cup B_1^{(2)})\}.$$
Considering the cardinality of $B_1^{(1)}, B_1^{(2)},B_1^{(3)}$, we have the following three cases. Recall that we have proved three different lower bounds for $|M|$, we will show in any one of the following three cases, a lower bounds for $|M|$ is larger than $4|B_1|$.

\noindent{\bf Case 1.} $3|B_1^{(1)}|\geq |B_1|$.

For any vertex $x \in C_1$ and $y \in V(G)\setminus
C$, the co-degree of $x$ and $y$ is at most $pn^2$ by Proposition
\ref{prop5}. Since the choices of $y$ is less than $n$, there are at
most $pn^3$ edges in $B_1^{(1)}$ containing $x$. Thus, we have
$|B_1^{(1)}|\leq pn^3|C_1|$. Bearing in mind that $|M|\geq
\frac{{\epsilon_1}^2\epsilon_2}{8\epsilon_3}pn^3|C_1|$, which is
shown in Claim \ref{lem18}, we can obtain that
$|M|\geq 13pn^3|C_1|\ge 13|B_1^{(1)}|>4|B_1|$.

\noindent{\bf Case 2.} $3|B_1^{(2)}|\geq |B_1|$.

Every $x \in C_2$ is in less than $\epsilon_2pn^3$ crossing edges of
$F$. Note that every edge in $B_1^{(2)}$ has at least one vertex in
$C_2$. If there exist more than $\epsilon_2pn^3$ edges with form
$xyuw$, where $ y\in A_1$, $ u\in A_2$, $ w\in A_3$, then we can
increase the number of crossing edges across the partition by moving
$x$ to $A_4$. Thus, in $F$, there are at most  $\epsilon_2pn^3$
edges of form $xyuw$, where $ y\in A_1$, $ u\in A_2$, $ w\in A_3$.
Furthermore, we can deduce that the total number of edges with form
$xyuw$, where $ y\in A_1$, $ u\in A_i$, $ w\in A_j$, $i\neq j \in
\{2,3,4\}$, is at most $3\epsilon_2pn^3$.

For any $x \in C_2$ and $y\in N_L(x)$, we  count the edges with form  $xyuw$, where  $ u,w\in A_2$.
Realize that if there exists a vertex $u\in A_2$, such that there $xyuw_1$, $xyuw_2$, $\ldots$,  $xyuw_{\epsilon_1pn}\in F$, where $u,w_1,\ldots,w_{\epsilon_1pn}\in A_2$, then for any $x' \neq x \in C_2$ and $y' \neq y \in A_1$, $x'y'w_iw_j$ is not in $F$, where $i\neq j$, otherwise, there exists a $T$ in $F$. Since we want to prove $|M|> 4|B_1|$, for every vertex $u\in A_2$, we assume the number of edges $xyuw$, where  $w\in A_2$ is less than ${\epsilon_1pn}$.
 So the number of edges $xyuw$, where $x\in C_2$, $y\in N_L(x)$, $u,w\in A_2$, is at most
$\frac{1}{2}\frac{n}{4}\epsilon_1 p n
\frac{n}{4}|C_2|=\frac{\epsilon_1}{32}n^3p|C_2|$.
Thus, the total number of edges $xyuw$, where $x\in C_2$, $y\in N_L(x)$, $u,w\in A_i$, $i=2,3,4$, is at most
$\frac{3\epsilon_1}{32}n^3p|C_2|$.

Now we count the edges $xyzu$, where $x,y,z\in C_2$, $u\notin C$.
Similarly, since $|C_2|\le |C|\le \epsilon_3n$, for every fixed $x\in C_2$ there are at most $\binom{\epsilon_3n}{2}$ choices of pair $(y,z)$. Combining with Proposition \ref{prop4}, the number of such edges is at most $\binom{\epsilon_3n}{2}2np|C_2|\le{{\epsilon_3}^2}n^3p|C_2|$.
Therefore, we have $|B_1^{(2)}|\leq 3\epsilon_2pn^3+\frac{3\epsilon_1}{32}n^3p|C_2|
+{{\epsilon_3}^2}n^3p|C_2|$.
So $|M|\ge\frac{1}{130}n^3p|C_2|>13|B_1^{(2)}|>4 |B_1|.$

\noindent{\bf Case 3.} $3|B_1^{(3)}|\geq |B_1|$.

Let $L'=L[C]\cup L[D]$. For any vertex $x\in D$, $d_L(x)\leq \epsilon_1n.$  And for any $y\in C$,  from Claim \ref{lem15}, we have  $d_{L'}(y)\leq  |C | \leq \epsilon_3n<\epsilon_1n.$
Applying Proposition \ref{prop5}, we have $|B_1^{(3)}|\leq pn^2|E(L')|$.  Combining with Claim \ref{lem17}, $ |M| \ge \frac{|E(L')|}{320\epsilon_1}pn^2 > 13|B_1^{(3)}|>4|B_1|$.

We complete the proof of Lemma \ref{lem13}. \qed

\subsection{Proof of Lemma \ref{lem14}}
Let $\epsilon=0.1$, $\alpha=0.35$, $\xi=0.001$,
$\gamma=\frac{1-\epsilon}{64}=0.146$,
$\alpha'=\frac{2\alpha}{1-\epsilon}=\frac{7}{9}$, $\varphi=0.0001$.
Note that for a balanced partition $\Pi=(A_1,A_2,A_3,A_4)$, $P(\Pi)
= \{(u, v) \in \binom{A_1}{2}: d_{\Pi}(u, v) <
\frac{\alpha}{32}p^2n^3\}$. By Propositions \ref{prop6} and
\ref{prop8}, we have $d(u,v)\le(1 +\epsilon)\frac{p^2}{6}n^3$ for every
pair $(u,v)$ of vertices, and $d_\Pi(v) \ge (1 -\epsilon)\frac{p}{64}n^3$ for
any vertex $v$. So $d_{\Pi}(u, v)
<\frac{2\alpha}{1-\epsilon}pd_\Pi(v)=\alpha'pd_\Pi(v)$ for every
pair $(u,v)$ in $P(\Pi)$. Let $\mathcal{A} $  be the event that for
$\delta > 0$, there exists a balanced cut $\Pi$ such that $q(G) \leq
|G[\Pi]| + |P(\Pi)|\delta n^3p^2$.  It is sufficient to prove
$Pr[\mathcal{A}] = o(1)$ for some $ \delta>0$. In fact, we will show
that $Pr[\mathcal{A}] = o(1)$  for $ \delta< \gamma\varphi/2$.   It
is well known that every graph contains an induced subgraph with at
least half of its edges. If we consider  $P(\Pi)$ as an edge set of
a graph, then we can derive that there exists a
subset $R$ of $P(\Pi)$ with $|R|\ge \frac{1}{2}|P(\Pi)|$, such that
$R$ is the edge set of some bipartite  graph.  By Proposition
\ref{prop11}, we have $d_{P(\Pi)}(v) < \frac{\xi}{p}$ for every
vertex $v$. Therefore, we have
\begin{equation}\label{eq4}
d_{R}(v) < \frac{\xi}{p}.
\end{equation}

Let $X$, $Y$ be disjoint subsets of $V$, $R$ be the edge set of a spanning subgraph of $[X, Y ]$ satisfying (\ref{eq4}), and
$f$ be a function from $X$ to $\{k \in \mathbb{N} : k \geq \gamma pn^3\}$. Denote by $\mathcal{B}(R,X, Y, f)$ the event that there
is a balanced cut $\Sigma$ of $G$ such that for every vertex $x$ in $X$ satisfying
\begin{equation}\label{eq5}
d_{\Sigma}(x)=f(x),\; R\subseteq P(\Sigma),
\end{equation}
and
$$q(G)\le|G[\Sigma]|+\varphi|R|\gamma n^3p^2.$$

We will show that there exists a constant $c$ such that
\begin{equation}\label{eq6}
Pr[\mathcal{B}(R,X, Y, f)] \le e^{-c|R|n^3p^2}.
\end{equation}
Note that if $\delta < \varphi\gamma/2$, then the event $\mathcal{A}$ implies event $ \mathcal{B}(R,X, Y, f)$ holds for some choice of $(R,X, Y, f)$. Realize that there are at most $\binom{\binom{n}{2}}{t}2^tn^{3t}$ choices of  $(R,X, Y, f)$ with $|R| = t$. Hence, by the union bound, if (\ref{eq6}) holds we have
$$Pr[\mathcal{A}]\le\sum\limits_{t > 0}{\binom{\binom{n}{2}}{t}2^tn^{3t}e^{-ctn^3p^2}}=o(1). $$

Now we  prove (\ref{eq6}). We choose all edges
 of $G$ according to the following three
 stages.
\begin{enumerate}
  \item[(a).] Choose the quadruples of vertices of $G$ that contain $x\in X$.
  \item[(b).] Choose the rest of the quadruples of vertices of $G$ except those belonging to $\bigcup\limits_{y \in Y} {[y,\bigcup\limits_{xy \in R} {L(x)} } ]$.
  \item[(c).] Choose the rest of the quadruples of vertices of $G$.
\end{enumerate}
Let $H$ be the subhypergraph of $G$ consisting of the edges
described in (a) and (b), and let $\Gamma$ be a
balanced cut of $H$ maximizing $|H[\Gamma]|$ among balanced cuts
$\Sigma$ satisfying (\ref{eq5}). For each $y\in Y$, set $M (y) =
\bigcup_{xy\in R}L_{\Gamma}(x)$. We have
\begin{equation}\label{eq8}
q(G)\ge |G[\Gamma]|=|H[\Gamma]|+\sum\limits_{y \in Y}{|G[y,M (y)]|}.
\end{equation}
Since $d_{\Gamma}(x)=f(x)\ge\gamma pn^3$,  for any two vertices $x$ and $x'$, we have $d_\Gamma(x, x')\le d(x, x')\le (1+\epsilon)\frac{p^2n^3}{6}\le \frac{pd_{\Gamma}(x)}{\gamma}$.
Also since $R$ satisfies (\ref{eq4}), so for each $y\in Y$ we
have
\begin{align*}
|M (y)|  &\ge \sum\limits_{xy \in R} {\left[d_{\Gamma}(x)-\sum\limits_{x\neq x' \in N_R(y)} {d_\Gamma(x, x')}\right] }\\
\null  &\ge \sum\limits_{xy \in R} {\left[d_{\Gamma}(x)-d_R(y)\max\limits_{x\neq x' \in N_R(y)} {d_\Gamma(x, x')} \right]}\\
\null  &\ge \sum\limits_{xy \in R} {\left[d_{\Gamma}(x)-\frac{\xi}{p}\frac{pd_{\Gamma}(x)}{\gamma} \right]}\\
\null  &\ge  (1-\frac{\xi}{\gamma}){\sum\limits_{xy \in R} {f(x)}}.
\end{align*}
Let $\mu=p\sum\limits_{y \in Y}{|M (y)| }$, then
$\mu\ge(1-\frac{\xi}{\gamma})p\sum \limits_{y \in Y}{\sum\limits_{xy
\in R} {f(x)}}.$ Using Lemma \ref{lem3}, we know that with
probability at least $1-e^{-c_{\epsilon}\mu}\ge1-e^{-c|R|n^3p^2}$,
for constant $c=c_{\epsilon}(\gamma-\xi)$, the sum  $\sum\limits_{y \in Y}{|G[y,M (y)]|}$
in (\ref{eq8}) is at least $(1 -\epsilon)\mu$.

On the other hand, for any balanced cut $\Sigma$, we have
$d_\Sigma(x, y) < \alpha'pd_\Sigma(x)$ for all $(x, y) \in
P(\Sigma)$. So for any balanced cut $\Sigma$ satisfying (\ref{eq5}),
we have
\begin{align*}
|G[\Sigma]|  &\le |H[\Sigma]|+\sum\limits_{y \in Y} {\sum\limits_{xy \in R} {d_\Sigma(x, y)} }\nonumber\\
\null  &\le |H[\Gamma]|+\sum\limits_{y \in Y} {\sum\limits_{xy \in R} {d_\Sigma(x, y)} }\nonumber\\
\null  &\le |H[\Gamma]|+\sum\limits_{y \in Y} {\sum\limits_{xy \in R} {\alpha'pd_\Sigma(x)}}\nonumber\\
\null  &\le |H[\Gamma]|+\alpha'p\sum\limits_{y \in Y} {\sum\limits_{xy \in R} {f(x)}}.
\end{align*}

Recall that, with probability at least $1-e^{-c|R|n^3p^2}$, the sum
$\sum\limits_{y \in Y}{|G[y,M (y)]|}$ in (\ref{eq8}) is at least $(1
-\epsilon)\mu$.  We have
$$q(G)-|G[\Sigma]|\ge[(1 -\epsilon)(1-\frac{\xi}{\gamma})-\alpha']p\sum\limits_{y \in Y}{\sum\limits_{xy \in R} {f(x)}}>\varphi|R|\gamma n^3p^2$$ holds with probability at least $1-e^{-c|R|n^3p^2}$,
which proves (\ref{eq6}). It is easy to check that the   constant
$\delta$ in  Lemma \ref{lem13} satisfies $ \delta< \gamma\varphi/2$,
so we can just let $\delta$ be the  constant $\delta$ in  Lemma
\ref{lem13}. \qed\\

\end{document}